\documentclass[11pt]{amsart}

\usepackage{amsfonts}
\usepackage{amssymb, latexsym, amsthm, amsmath, verbatim}
\usepackage{indentfirst}

\numberwithin{equation}{section}
\theoremstyle{definition}
\newtheorem{Thm}[equation]{Theorem}

\newtheorem{Cor}[equation]{Corollary}
\newtheorem{Lem}[equation]{Lemma}

\newtheorem{Exa}[equation]{Example}
\newtheorem{Rmk}[equation]{Remark}

\makeatletter
\def\imod#1{\allowbreak\mkern5mu{\operator@font mod}\,\,#1}
\makeatother

\allowdisplaybreaks

\begin{document}

\title[Divisibility]{Divisibility Properties for Weakly Holomorphic Modular Forms with Sign Vectors}
\author[Yichao Zhang]{Yichao Zhang}
\address{Department of Mathematics, the University of Arizona, Tucson, AZ 85721-0089}
\email{yichaozhang@math.arizona.edu, zhangyichao2002@gmail.com}
\date{}

\begin{abstract}
In this paper, we prove some divisibility results for the Fourier coefficients of reduced modular forms of sign vectors. More precisely, we generalize a divisibility result of Siegel on constant terms when the weight is non-positive, which is related to the weight of Borcherds lifts when the weight is zero. By considering Hecke operators for the spaces of weakly holomorphic modular forms with sign vectors, and obtain divisibility results in an ``orthogonal" direction on reduced modular forms. 
\end{abstract}

\maketitle

\section*{Introduction}

\noindent
Weakly holomorphic modular forms for Weil representations has become an active field of research since Borcherds \cite{borcherds1998automorphic} discovered the theory of automorphic products using regularized theta lifting. Roughly speaking such a lift sends weakly holomorphic modular forms for Weil representations to automorphic forms on orthogonal groups. In order to concretely view such a lift, Bruinier and Bundschuh \cite{bruinier2003borcherds} constructed an isomorphism between certain spaces of (scalar-valued) weakly holomorphic modular forms and certain spaces of weakly holomorphic modular forms for Weil representations when the level $N=p$ is an odd prime. Such an isomorphism was recently generalized by the author to more general level $N$ (\cite{zhang2014isomorphism, zhang2013zagier}) and notions such as sign vectors and reduced modular forms were introduced. Such an isomorphism proves to be useful and has important applications. For example, on the set of reduced modular forms, Zagier duality \cite{zhang2013zagier} was obtained, a duality between Fourier coefficients of integral weight $k$ modular forms and that of weight $2-k$ modular forms. See also the application on automorphic correction of hyperbolic Kac-Moody algebras (\cite{kim2012rank,kim2013weakly}).

Because of the isomorphism to modular forms for Weil representations of $\text{SL}_2(\mathbb Z)$, such modular forms of sign vectors are essentially of level one and consequently properties of classical modular forms of level one should also hold on these modular forms. For example, the holomorphicity at $\infty$ of a weakly holomorphic modular form with sign vector determines its holomorphicity at other cusps. In this paper, we consider some divisibility properties of Fourier coefficients of reduced modular forms with sign vectors. Following an argument of Duke and Jenkins \cite{duke2008zeros}, we first extend a result of Siegel \cite{siegel1969berechnung} on the constant terms of reduced modular forms of level one and of weight $k\leq 0$ (Theorem 4.2). We note that Siegel's divisibility result actually holds for all reduced modular forms $f_m$, not just for $f_{-\ell-1}$, where $\ell$ is the dimension of the cuspform space of weight $2-k$. 

Such divisibility becomes very interesting when $k=0$, for which we have the Borcherds's lift and the constant term of $f_m$ represents the weight of the resulting Borcherds's lift $\Psi_{f_m}$ (a Hilbert modular form if $N$ is a fundamental discriminant). When $N=5$ and $\epsilon=+1$, the reduced modular forms, computed by Bruinier and Bundschuh \cite{bruinier2003borcherds}, begin with
\begin{align*}f_{-1}&=q^{-1} + 5 + 11q - 54q^4 + 55q^5 + 44q^6 - 395q^9 + 340q^{10}+O(q^{11}),\\
f_{-4}&=q^{-4} + 15 - 216q + 4959q^4 + 22040q^5 - 90984q^6 + 409944q^9 + 1388520q^{10}+O(q^{11}),\\
f_{-5}&=\frac{1}{2}q^{-5} + 15 + 275q + 27550q^4 + 43893q^5 + 255300q^6 + 4173825q^9 + 4807100q^{10}+O(q^{11}).
\end{align*}
We may easily see that the constant term of $f_m$ is divisible by $5$ for all $m$, and consequently the weights of the Borcherds lift $\Psi_{f_m}$ are all divisible by $5$. We note that this can also be seen from the Zagier duality since the obstruction space is trivial. However, it seems that there is no obvious way of seeing this from the construction of $f_m$. One may hope for more such divisibility results, but it turns out that above particular example is the only possible divisibility that can be thus obtained on the weights of Borcherds lifts. In other words, above method does not apply to more cases  when $k=0$. This is related to generalized Bernoulli numbers, whose denominators are trivial in other cases (see Remark 4.5 for details).

We then proceed in a direction ``orthogonal" to Siegel's result and its generalization. The corresponding divisibility properties happen inside individual reduced modular forms of weight $k\leq 0$. For level one weakly holomorphic modular forms of weight $k$ with $2-k=4,6,8,10,14$, such results were obtained by Duke and Jenkins \cite{duke2008zeros}, using the fact that the weight $2-k$ cuspform space, the obstruction space, is trivial. As Duke and Jenkins did, assuming certain integrality on the Fourier coefficients of reduced modular forms, we consider the Hecke operators on such spaces of modular forms and then derive a few divisibility results (Theorem 4.7, 4.9). The assumption on trivial obstruction spaces is also needed, but the situation is subtler because of the existence of more than one sign vectors. Several examples will be presented in order to make it clear. Note that Bruinier and Stein \cite{bruinier2010weil} constructed Hecke operators for modular forms for Weil representations, and in our particular case, their Hecke operators belong to a subalgebra of the level $N$ Hecke operators via the isomorphism. We finally see another way of obtaining such divisibility results by applying the differential operator $D^{1-k}$ with $D=q\frac{d}{dq}$ (Theorem 4.13), where divisibility by $p\mid N$ is included. These two methods have vanishing assumptions on different $\epsilon$-subspaces.

Here is the layout of this paper. We recall necessary notions in the first section, and give an easier proof of Zagier duality in the second section. In Section 3, we briefly consider Hecke operators for subspaces with sign vectors. In the last section, we prove divisibility results, in two directions, for Fourier coefficients of reduced modular forms.

\subsection*{Acknowledgments} 

\section{Preliminaries on Modular Forms}\label{scalar}
\noindent
We set up the notations and recall some results in this section. See \cite{miyake2006modular} for the general theory on modular forms and \cite{zhang2014isomorphism,zhang2013zagier} for results on modular forms of sign vectors.

We shall fix a primitive quadratic Dirichlet character $\chi$ and denote its conductor by $N$. Assume that $\chi$ decomposes to $\chi=\prod_{p\mid N}\chi_p$. Even though most results hold for the degenerate case when $\chi$ is trivial and $N=1$, we shall assume that $N>1$. 
Denote by $N_p$ the largest $p$-power in the factorization of $N$.
We shall write $p^\nu||m$ if $p^\nu\mid m$ and $p^{\nu+1}\nmid m$, so $N_p||N$. Note that $N_p=p$ if $2<p\mid N$ and $N_2=1,4$ or $8$.

Let $k\in\mathbb Z$. We denote $A(N,k,\chi)$ the space of weakly holomorphic modular functions of level $N$, weight $k$ and character $\chi$; namely, the space of functions $f$ that are holomorphic on the upper half plane, meromorphic at cusps, and
\[(f|_kM)(\tau)=\chi(d)f(\tau),\quad\text{for all } M=\begin{pmatrix}a&b\\c&d\end{pmatrix}\in\Gamma_0(N).\]
Here $f|_kM(\tau)=\text{det}(M)^\frac{k}{2}(c\tau+d)^{-k}f(\tau)$ for $M\in\text{GL}_2^+(\mathbb R)$ is the slash-$k$ operator. Let $M(N,k,\chi)$ and $S(N,k,\chi)$ be the subspace of holomorphic forms and that of cuspforms respectively.

For each sign vector $\epsilon=(\epsilon_p)_{p\mid N}$, that is $\epsilon_p\in\{\pm 1\}$, we impose the $\epsilon$-condition and obtain the subspace $A^\epsilon(N,k,\chi)$ for each  of $A(N,k,\chi_D)$ as follows:
\[A^\epsilon(N,k,\chi_D)=\left\{f=\sum_n a(n)q^n\in A(N,k,\chi_D)\left|
 \begin{split}a(n)=0 \text{ if } \chi_p(n)=-\epsilon_p \text{ for some }p\mid N\end{split}\right.\right\}.\]
Denote $M^\delta(N,k,\chi)=M(N,k,\chi)\cap A^\delta(N,k,\chi)$ and similarly we have $S^\delta(N,k,\chi)$.
Recall that the dual sign vector $\epsilon^*=(\epsilon_p^*)_{p\mid N}$ is defined by $\epsilon^*_p=\chi_p(-1)\epsilon_p$.

If $8\nmid N$, the pair $(\chi,\epsilon)$ determines a discriminant form $D$, that is, a finite abelian group with a $\mathbb Q/\mathbb Z$-valued nondegenerate quadratic form. Explicitly, we fix a \emph{Jordan decomposition} $D=\bigoplus_{p\mid N}D_p$ and the \emph{Jordan component} $D_p$ is determined as follows: if $p$ is an odd prime divisor of $N$, then $D_p=\mathbb Z/p\mathbb Z$ with the quadratic form on $D_p$ given by $Q(x)=\frac{ax^2}{p}$ with $\chi_p(aN/p)=\epsilon_p$; if $p=2$ and $4|| N$, then 
\[D_2=\mathbb Z/2\mathbb Z\times \mathbb Z/2\mathbb Z,\quad\text{with } Q((1,0))=Q((0,1))=\frac{\epsilon_2\chi_2(N/4)}{4}.\]
The case when $8|| N$ is more complicated. Actually, if $8|| N$ and $p=2$, there will be two discriminant forms for $D_2$. Namely,
\[D_2=\mathbb Z/2\mathbb Z\times \mathbb Z/4\mathbb Z,\quad\text{with } Q((1,0))=\frac{t_1}{4}, Q((0,1))=\frac{t_2}{8},\]
with $t_1\in \{\pm 1\}$, $t_2\in\{\pm 1,\pm 3\}$ such that
\[\chi_2(-1)=e(-(t_1+t_2)/8),\quad \chi_2(t_2N/8)=\epsilon_2.\]
One can see easily that these two possible $D_2$ are actually isomorphic, justifying the fact that Jordan components and indecomposable components are not unique in general. So if $8\mid N$, we will fix $D$ to be either of the two possible discriminant forms above.
Conversely, each of such discriminant forms determines uniquely a pair $(\chi,\epsilon)$ (see \cite{scheithauer2009weil} or \cite{zhang2013zagier}).

Every discriminant form $D$ can be realized as $M'/M$, where $M$ is an even lattice, $M'$ is the dual lattice of $M$ and the quadratic form of $D$ is that of $M$ modulo $\mathbb Z$. With such $M$, the \emph{signature} of $D$, denoted by $r$, is the signature of $M$ modulo $8$. Throughout this paper, we assume that $r$ is even. This assumption implies that the Weil representation constructed from $D$ for $\text{Mp}_2(\mathbb Z)$ factors through $\text{SL}_2(\mathbb Z)$.  The dual discriminant form $D^*$ is define to be the same abelian group with the quadratic form $-Q(\cdot)$. 

Let $\rho_D$ be the Weil representation of $SL_2(\mathbb Z)$ on $\mathbb C[D]$; that is, if $\{e_\gamma:\gamma\in D\}$ is the standard basis for the group algebra $\mathbb C[D]$, then the action
\begin{align*}
\rho_D(T)e_\gamma &= e(q(\gamma))e_\gamma,\\
\rho_D(S)e_\gamma &=\frac{i^{-\frac{r}{2}}}{\sqrt{N}}\sum_{\delta\in D}e(-(\gamma,\delta))e_\delta,
\end{align*}
defines the unitary representation $\rho_D$ of $\text{SL}_2(\mathbb Z)$ on $\mathbb C[D]$. Here and after, $e(z)=e^{2\pi i z}$ for $z\in\mathbb C$.

Let $\mathcal A(k,\rho_D)$ be the space of modular forms of weight $k$ and type $\rho_D$. That is, $F=\sum_\gamma F_\gamma e_\gamma\in\mathcal A(k,\rho_D)$ if $F|_kM:=\sum_\gamma (F_\gamma|_kM)e_\gamma=\rho_D(M)F$ for any $M\in\text{SL}_2(\mathbb Z)$, $F_\gamma$ is holomorphic on the upper half plane and $F_\gamma=\sum_{n\in q(\gamma)+\mathbb Z}a(\gamma,n)q^n$ with at most finitely many negative power terms. Let $\mathcal M(k,\rho_D)$ and $\mathcal S(k,\rho_D)$ denote the space of holomorphic forms and the space of cusp forms respectively.
We shall also need $\mathcal A^\text{inv}(k,\rho_D)$, the subspace of modular forms that are invariant under $\text{Aut}(D)$. Analogously we have $\mathcal M^{\text{inv}}(k,\rho_D)$ and $\mathcal S^{\text{inv}}(k,\rho_D)$.

For convenience, we quote the isomorphism theorem in \cite{zhang2013zagier} as follows:

\begin{Thm}[{\cite[Theorem 3.3]{zhang2013zagier}}] Assume that $D$ and $(\chi,\epsilon)$ correspond to each other as described above.
There exists an isomorphism between $A^\epsilon(N,k,\chi)$ and $\mathcal A^\text{inv}(k,\rho_D)$, which sends $f=\sum_na(n)q^n\in A^\epsilon(N,k,\chi)$ to $F=\sum_\gamma F_\gamma e_\gamma$ with
\[F_\gamma(\tau)=s(NQ(\gamma))\sum_{n\equiv NQ(\gamma)\imod N\mathbb Z}a(n)e\left(n\tau/N\right)=\sum_{n\equiv NQ(\gamma)\imod N\mathbb Z}s(n)a(n)e\left(n\tau/N\right).\]
\end{Thm}

Here for each $m\imod N$, $s(m)=2^{\omega((m,N))}$ and $\omega(m)$ the number of distinct prime divisors of $m$. Throughout this paper, by \emph{the isomorphism}, we shall always mean the one in Theorem 1.1.

\section{Zagier Duality}

In \cite{zhang2013zagier}, we proved the Zagier duality for reduced modular forms and obtained the complete grids. In order to avoid some computational difficulty, we assumed that $8\nmid N$. In this section, for completeness and to remove the assumption that $8\nmid N$, we prove the Zagier duality in a different way. Roughly speaking, we pass to the spaces $\mathcal A^{\text{inv}}(k,\rho_D)$ and utilize a pairing therein.

We first recall that a \emph{reduced modular forms of order $m$} in $A^\epsilon(N,k,\chi)$ for each $m\in\mathbb Z$, denoted by $f_m$ or $f_m^\epsilon$ if it exists, is the modular form of the form $f_m=\sum_na(n)q^n=\frac{1}{s(m)}q^m+O(q^{m+1})$ such that for each $n>m$ with $a(n)\neq 0$, there does not exist $g\in A^\epsilon(N,k,\chi)$ such that $g=q^n+O(q^{n+1})$. Such notion is a generalization of the modular forms in a \emph{Miller basis} for level one holomorphic modular form spaces.

\begin{Lem}
For any integers $k_1,k_2$, we have following pairing
\[\mathcal A(k_1,\rho_D)\times \mathcal A(k_2,\rho_{D^*})\rightarrow A(1,k_1+k_2,1),\] 
given by
\[\langle F,G\rangle=\sum_{\gamma\in D}F_\gamma G_\gamma,\quad F=\sum_{\gamma\in D}F_\gamma e_\gamma,G=\sum_{\gamma\in D}G_\gamma e_\gamma.\] In particular, $\langle F|_{k_1}M,G|_{k_2}M\rangle=\langle F,G\rangle$ for all $M\in\text{SL}_2(\mathbb Z)$.
\end{Lem}
\begin{proof}
The proof is elementary. Note that we only have to prove the transformation formula for the generators $S$ and $T$. For $T$, this is clear by noting that the two discriminant forms are dual to each other. For $S$, we shall also need the fact that the bilinear form for the discriminant form is nondegenerate. We omit the details.
\end{proof}

The duality concerns the weights $k$ and $2-k$. Therefore, if $k\neq 1$, without loss of generality, we may assume that $k\leq 0$.

\begin{Thm}[{\cite[Theorem 5.7 ]{zhang2013zagier}}] Let $k\leq 0$ be an integer, $\epsilon=(\epsilon_p)$ be any sign vector and let $\epsilon^*$ be the dual sign vector. Assume $m,d\in\mathbb Z$. Assume that both of the reduced modular forms \[f_{m}=\sum_{n\in\mathbb Z}a_{m}(n)q^n\in A^\epsilon(N,k,\chi) \quad \text{\ and\ }\quad  g_{d}=\sum_{n\in\mathbb Z}b_{d}(n)q^n\in A^{\epsilon^*}(N,2-k,\chi)\] exist (hence $m<0$). Then we  have $a_{m}(-d)=-b_{d}(-m)$.
\end{Thm}
\begin{proof}
Let $F\in \mathcal A(k,\rho_D)$ and $G\in\mathcal A(k,\rho_{D^*})$ be the corresponding vector-valued modular forms for $f_m$ and $g_d$ under the isomorphism.

By Lemma 2.1, we see that $\langle F,G\rangle\in A(1,2,1)$. In particular, $\langle F,G\rangle d\tau$ is a meromorphic $1$-form on the compact Riemann surface $X(1)$. It follows that the sum of residues of $\langle F,G\rangle d\tau$ vanishes. Since $F$ and $G$ are holomorphic on $\mathbb H$, the residue at $\infty$ vanishes. It is clear that the residue at $\infty$ of $\langle F,G\rangle d\tau$ is given by the constant term of $\langle F,G\rangle$, which is equal to
\[\frac{1}{2\pi i}\sum_{\gamma\in D}\sum_{n\equiv NQ(\gamma)\imod N}s(n)a_m(n)b_d(-n)=\frac{1}{2\pi i}\sum_{n\in\mathbb Z}s(n)a_m(n)b_d(-n),\]
by the isomorphism.
We then have
\[0=\sum_{n\in\mathbb Z}s(n)a_m(n)b_d(-n)=\sum_{m\leq n\leq -d}s(n)a_m(n)b_d(-n)=a_m(-d)+b_d(-m)+\sum_{m<n<-d}s(n)a_m(n)b_d(-n).\]
By \cite[Lemma 5.5]{zhang2013zagier} we must have $m\neq -d$ and if $m>-d$ and $a_{m}(-d)=-b_{d}(-m)=0$. Therefore, we only need to treat the case when $m<-d$, in which case we have
\[0=a_m(-d)+b_d(-m)+\sum_{m<n<-d}s(n)a_m(n)b_d(-n)=a_m(-d)+b_d(-m),\]
by \cite[Lemma 5.6]{zhang2013zagier}. This completes the proof.
\end{proof}

\begin{Rmk}
The bases $\{f_m\}$ and $\{g_d\}$ form the complete \emph{grids} for Zagier duality in the sense of \cite[Remark 5.8]{zhang2013zagier}. Roughly speaking, such duality exhausts, of course in pairs, all nonzero Fourier coefficients (except the leading coefficients).
\end{Rmk}

\begin{Exa}
We present an example when $8\mid N$. Consider $N=8$, $k=0$ and $\chi=\left(\frac{2}{\cdot}\right)$ (we may also choose the other character $\left(\frac{-2}{\cdot}\right)$). Let $\epsilon=+1$, so $\epsilon^*=+1$. We have the following basis of reduced modular forms $\{f_m\}$ in $A^\epsilon(8,2,\chi)$:
\[\arraycolsep=3.7pt\def\arraystretch{1.1}
\begin{array}{rrrrrrrrrrrr}
f_0 &=& 1/2 &- 2q &- 3q^2 &- 5q^4 &- 2q^6 &- 16q^7& - 9q^8 &- 14q^9 &+ O(q^{10})\\
f_{-1} &=& q^{-1} &- 2q &- 8q^2& + 16q^4& + 48q^6& - 7q^7 &- 96q^8 &+ 18q^9 &+O(q^{10}),\\
f_{-2} &=&\frac{1}{2}q^{-2}& - 4q& + 3q^2& - 28q^4& + 72q^6& + 224q^7& - 168q^8& - 540q^9& +O(q^{10}),\\
f_{-4} &=&\frac{1}{2}q^{-4}&+ 4q& - 14q^2& - 89q^4& - 420q^6& + 1568q^7& - 1460q^8& + 5148q^9& +O(q^{10}),\\
f_{-6} &=&\frac{1}{2}q^{-6}&+8q & + 24q^2& - 280q^4& + 1708q^6& - 7616q^7& - 8016q^8& + 31800q^9 & +O(q^{10}),\\
f_{-7} &=& q^{-7}&- q& + 64q^2& + 896q^4& - 6528q^6& - 128q^7& - 34048q^8& - 18q^9 & +O(q^{10}),\\
& \vdots&&&&&&&&&&
\end{array}
\]
For the dual space $A^{\epsilon^*}(8,0,\chi)$, the basis of reduced modular forms $\{g_d\}$ are:
\[\arraycolsep=1.6pt\def\arraystretch{1.1}
\begin{array}{rrrrrrrrrrrr}
g_{-1} &=&q^{-1}&+ 2 &+ 2q &+ 4q^2 &- 4q^4 &- 8q^6 &+ q^7 &+ 12q^8 &- 2q^9&+O(q^{10}),\\
g_{-2} &=&\frac{1}{2}q^{-2}&+ 3 &+ 8q &- 3q^2 &+ 14q^4 &- 24q^6 &- 64q^7 &+ 42q^8 &+ 120q^9&+O(q^{10}),\\
g_{-4} &=&\frac{1}{2}q^{-4}&+ 5 &- 16q &+ 28q^2 &+ 89q^4 &+ 280q^6 &- 896q^7 &+ 730q^8 &- 2288q^9&+O(q^{10}),\\
g_{-6} &=&\frac{1}{2}q^{-6}&+2 &- 48q &- 72q^2 &+ 420q^4 &- 1708q^6& + 6528q^7 &+ 6012q^8 &- 21200q^9&+O(q^{10}),\\
g_{-7} &=&q^{-7}&+ 16 &+ 7q &- 224q^2 &- 1568q^4 &+ 7616q^6 &+ 128q^7 &+ 29792q^8 &+ 14q^9&+O(q^{10}),\\
& \vdots &&&&&&&&&
\end{array}\]
The duality can be detected from these two tables, by ignoring the first columns and viewing one table horizontally and the other vertically.
\end{Exa}

\section{Hecke Operators and the Differential Operator $D^{1-k}$}\label{vector}

Bruinier and Stein \cite{bruinier2010weil} constructed Hecke operators for modular forms associated to Weil representations. Via the isomorphism, their set of Hecke operators correspond to a subset of the Hecke algebra for $A(N,k,\chi)$ in our setting. With such operators, we will prove some divisibility results for reduced modular forms in next section.

From now on, we denote a positive integer by $r$. We recall the Hecke operators $T(r)$ on $A(N,k,\chi)$ with $(r,N)=1$:  it acts on $f=\sum_na(n)q^n$ by
\[f|_kT(r)=\sum_{n}b(n)q^n,\] with \[b(n)=\sum_{0<d\mid (r,n)}\chi(d)d^{k-1}a(rn/d^2).\]

Denote $R_0$ the subset of $\mathbb Z_{>0}$
\[R_0=\{r\in\mathbb Z_{>0}\colon \chi_p(r)=1 \text{ for each } p\mid N\}.\]
For a sign vector $\epsilon$ and the subspace $A^\epsilon(N,k,\chi)$, we consider the subalgebra $\mathcal R_0$ of the Hecke algebra that is generated by
$\{T(r)\colon r\in R_0\}$. 

\begin{Lem}
The Hecke algebra $\mathcal R_0$ acts on $A^\epsilon(N,k,\chi)$ for each $\epsilon$.
\end{Lem}
\begin{proof}
Suppose $f=\sum_na(n)q^n\in A^\epsilon(N,k,\chi)$ and $r\in R_0$. Note first that $f|_kT(r)\in A(N,k,\chi)$. Now assume $f|_kT(r)=\sum_nb(n)q^n$ where 
\[b(n)=\sum_{0<d\mid (r,n)}\chi(d)d^{k-1}a(rn/d^2).\]
Now if for some $p\mid N$ we have $\chi_p(n)=-\epsilon_p$, then for each $d\mid (r,n)$ we must have $\chi_p(rn/d^2)=-\epsilon_p$ because $\chi_p(r)=1$. This implies that $a(rn/d^2)=0$ since $f\in A^\epsilon(N,k,\chi)$. Therefore, $b(n)=0$ and $f|_kT(r)\in A^\epsilon(N,k,\chi)$. This finishes the proof.
\end{proof}

We warn here that $T(p)$ when $p\mid N$ does not act on $A^\epsilon(N,k,\chi)$ even though $f|_kT(p)$ is a scaler multiple of $f|_k\eta_p$ and $\eta_p$ is an involution on $A(N,k,\chi)$ (\cite[Lemma 2.5]{zhang2013zagier} or \cite[Corollary 4.12]{zhang2014isomorphism}).

\begin{Rmk}
The Hecke operators constructed by Bruinier and Stein \cite{bruinier2010weil} generate a subalgebra of $\mathcal R_0$ in our setting ($k\in\mathbb Z$ and $\chi$ is a primitive quadratic Dirichlet character with conductor $N$). More precisely, the Hecke operators $T(r^2)^*$ ($(r,N)=1$) on the vector-valued modular form space becomes $T(r^2)$ for the scalar-valued modular forms. 
\end{Rmk}

\begin{Lem}
Let $r$ be any positive integer with $(r,N)=1$ and $\epsilon=(\epsilon_p)$ be a sign vector. Set $\epsilon'=(\epsilon_p')$ with $\epsilon_p'=\epsilon_p\chi_p(r)$. Then $T(r)$ maps $A^\epsilon(N,k,\chi)$ into $A^{\epsilon'}(N,k,\chi)$.
\end{Lem}
\begin{proof}
Suppose $f=\sum_na(n)q^n\in A^\epsilon(N,k,\chi)$ and assume $f|_kT(r)=\sum_nb(n)q^n$ where 
\[b(n)=\sum_{0<d\mid (r,n)}\chi(d)d^{k-1}a(rn/d^2).\]
Now if for some $p\mid N$ we have $\chi_p(n)=-\epsilon_p'$, then for each $d\mid (r,n)$ we must have $\chi_p(rn/d^2)=\chi_p(rn)=-\chi_p(r)^2\epsilon_p$. This implies that $a(rn/d^2)=0$ since $f\in A^\epsilon(N,k,\chi)$. Therefore, $b(n)=0$ and $f|_kT(r)\in A^{\epsilon'}(N,k,\chi)$. 
\end{proof}

\begin{Exa}
Consider the case $N=15$ and $\chi=\left(\frac{\cdot}{15}\right)$. There are four distinct sign vectors $\epsilon$:
\[\epsilon_1=(-1,-1),\quad \epsilon_2=(1,-1),\quad \epsilon_3=(-1,1),\quad \epsilon_4=(1,1).\]
Among them, $\epsilon_1$ and $\epsilon_2$ are dual to each other, and $\epsilon_3$ and $\epsilon_4$ are dual to each other.
We consider the cuspform space when $k=3$. We know that $S(15,3,\chi)=\mathbb Cg_1+\mathbb Cg_2$, with
\begin{align*}g_1&=q - 3q^4 - 3q^6 + 9q^9 + 5q^{10} +O(q^{15})\in S^{\epsilon_4}(15,3,\chi),\\
g_2&=q^2 - 3q^3 + 5q^5 - 7q^8 + 9q^{12}+O(q^{15})\in S^{\epsilon_1}(15,3,\chi).
\end{align*} Since both spaces are one-dimensional, $g_1$ and $g_2$ are common eigenfunctions for Hecke operators in $\mathcal R_0$. To verify this numerically, let us add more terms to $g_1=\sum_na(n)q^n$:
\[g_1=q - 3q^4 - 3q^6 + 9q^9 + 5q^{10} - 15q^{15} + 5q^{16} - 22q^{19} + 21q^{24} + 25q^{25} + 2q^{31} - 14q^{34} - 27q^{36}\]\[ - 35q^{40} + 34q^{46} + 49q^{49}+ 42q^{51} - 27q^{54} + 45q^{60} - 118q^{61} + 13q^{64} - 102q^{69} + 66q^{76} + O(q^{77}).\]
Now $4\in R_0$ and we should have $a(4n)=a(4)a(n)$ if $(2,n)=1$. This is clear from above Fourier expansion. Similarly, $19\in R_0$ and one can see that $a(19n)=a(19)a(n)$ if $19\nmid n$.

Moreover, one can verify easily that $g_1|T(2)=g_2$, $g_2|T(2)=g_1$ and $T(2)$ interchanges the $\epsilon_4$-subspace and the $\epsilon_1$-subspace.
\end{Exa}

We finally recall the differential operator $D=q\frac{d}{dq}$. It was treated in many places; for example, one may refer to Zagier's paper \cite{zagier1994modular} for details. We note that in general $D$ destroys the modularity, but it is well-known that when $k\leq 0$,
$D^{1-k}\colon A(N,k,\chi)\rightarrow A(N,2-k,\chi)$. Actually $D^{1-k}$ is a special case of the \emph{Rankin-Cohen bracket}.
Clearly if $f=\sum_na(n)q^n$, then
$D^{1-k}f=\sum_nn^{1-k}a(n)q^n$. In particular, the constant term of $D^{1-k}f$ vanishes. Moreover, one sees that $D^{1-k}(A^\epsilon(N,k,\chi))\subset A^\epsilon(N,2-k,\chi)$.

\section{Divisibility of Fourier Coefficients}

From now on, we shall assume that for any reduced modular form \[f_m=\sum_na_m(n)q^n\in A^\epsilon(N,k,\chi),\] the modular form $s(m)f_m$ has integral Fourier coefficients. Namely $s(m)a_m(n)\in\mathbb Z$ for any $n\in\mathbb Z$.  We remark that such integrality for each fixed reduced modular form is easy to verify numerically by Sturm's theorem (\cite{sturm1987congruence}, \cite[Corollary 3.2 ]{kim2013weakly}). Such integrality holds so far for all of the numerically examples that we have computed except in the case of weight $k=0$. For example, when $N=17$ and $\epsilon=+1$, $s(m)f_m$ may contain half-integral constant terms (see Mayer's computation \cite[Section 5.1.3]{mayer2007hilbert}). Such exception seems to be related to the fact that constant functions are modular functions for $\Gamma_1(N)$. Even in case of such exception, one can adjust (or just ignore) the divisibility by $2$-powers.

It is also noteworthy that a different type of integrality $s(n)a(n)\in\mathbb Z$ was needed and treated partially in \cite{kim2013weakly, zhang2013zagier}. Actually, for each level $N$ and weight $k$, the integrality boils down to that of finitely many reduced modular forms (\cite[Lemma 6.1]{zhang2013zagier}) and for each fixed reduced modular form, this type of integrality can be verified numerically using Sturm's Theorem.

We shall also assume the existence whenever we write $f_m$ in this section. Alternatively, if $f_m$ does not exist, we may just understand $f_m=0$.

\subsection{Constant Terms and a Result of Siegel}
We first generalize a result of Siegel to higher level reduced modular forms.
Siegel considered the constant terms of weakly holomorphic modular forms of level one and of negative weight and proved the his divisibility result (see \cite{siegel1969berechnung,duke2008zeros}). 
More precisely, if $f_m=\sum_{n}a_m(n)q^n$ denote a reduced modular form of weight $2-k<0$, level one and if $\ell$ is the dimension of the space of cuspforms of weight $k$ (the dual space), then $p\mid a_{-l-1}(0)$ whenever $(p-1)\mid k$. By the duality, this amounts to saying that
\[p\mid a_0(\ell+1),\quad \text{ whenever } (p-1)\mid k.\]
For example, when $k=12$, we have $\ell=1$ and
\[f_{0}=1+196560q^2+16773120q^3+398034000q^4+O(q^5).\]
Here $(p-1)\mid 12$ if $p=2,3,5,7,13$ and we have
\[a_0(2)=196560=2^4\cdot 3^3\cdot 5\cdot 7\cdot 13.\]

We pass to the case $k\geq 2$ via the duality. To generalize this result, we need the Dirichlet $L$-function $L(s,\chi)$ at negative integers.
\begin{Lem}
Let $k\geq 2$ and
\[f_0=\sum_na_0(n)q^n\in A^\epsilon(N,k,\chi)\]
be the reduced modular form of order $0$. Then $r\mid s(0)a_0(n)$ if $n>0$, where $r$ is determined by
\[L(1-k,\chi)=\frac{s}{r},\quad r,s\in\mathbb Z, (r,s)=1.\]
\end{Lem}
\begin{proof}
By \cite[Lemma 4.3]{zhang2013zagier}, the Eisenstein space in $M^\epsilon(N,k,\chi)$ is one-dimensional and generated by $E^\epsilon$, where 
\[E^\epsilon=\frac{1}{s(0)L(1-k,\chi)}\sum_{0<m\mid N}\epsilon_mE_m.\]
For the meaning of notations, see \cite[Section 4]{zhang2013zagier} or \cite[Chapter 4 ]{diamond2005first}. This implies the existence of $f_0$. For $E^\epsilon$, we just need the fact that when $m>1$, $E_m$ and $E_1-L(1-k,\chi)$ vanish at $\infty$ and has integral coefficients. In other words, $s(0)L(1-k,\chi)E^\epsilon\in L(1-k,\chi)+q\mathbb Z[[q]]$.

Since reduced modular forms form a basis, we see that for any $f=\sum_n a(n)q^n\in M^\epsilon(N,k,\chi)$,
\[f=\sum_{m\geq 0}a(m)s(m)f_m,\]
where if $f_m$ does not exists, we understand $f_m=0$. Obviously, the right side is a finite sum.
Now we compare coefficients of $q^n$ for $n>0$ on both sides, and we obtain
\[a(n)=\sum_{m\geq 0}a(m)s(m)a_m(n).\] We then substitute $E^\epsilon=\sum_nB(n)q^n=\frac{1}{s(0)}+O(q)$ for $f$ in above equation, so
\[B(n)=\sum_{m\geq 0}B(m)s(m)a_m(n),\quad\text{or}\quad B(0)s(0)a_0(n)=-\sum_{m>0}B(m)s(m)a_m(n)+B(n).\] Since if $n>0$, $s(0)L(1-k,\chi)B(n)\in\mathbb Z$, and $s(m)a_m(n)\in\mathbb Z$ for any $m,n$, we must have
\[L(1-k,\chi)s(0)a_0(n)=L(1-k,\chi)s(0)B(0)s(0)a_0(n)\in\mathbb Z.\]
This finishes the proof.
\end{proof}

\begin{Thm}
We keep notations in Lemma 4.1 and assume that $N$ is a prime power and $n>0$.
\begin{enumerate}
\item Assume $N=p>2$, $p^\nu||k$ and $t\imod p$ is a primitive root. Then $p^{\nu+1}\mid s(0)a_0(n)$ if
\begin{itemize}
	\item $\chi(t)=1$ and $(p-1)\mid k$, or
	\item $\chi(t)=-1$ and $(p-1,k)=\frac{p-1}{2}$.
\end{itemize} 
\item If $N=4$ and $2\nmid k$, then $2\mid s(0)a_0(n)$.
\end{enumerate}
\end{Thm}
\begin{proof}
By Lemma 4.1, we only have to show that in these two cases, the denominator of $L(1-k,\chi)$ is $p^{\nu+1}$ and $2$ respectively. It is well-known that $L(1-k,\chi)=-\frac{B_{k,\chi}}{k}$, where $B_{k,\chi}$ is the $k$-th generalized Bernoulli number associated to the character $\chi$. By a result of Carlitz (\cite[Theorem 1,3]{carlitz1959arithmetic}), we know that if $N=4$ and $2\nmid k$, then $-\frac{B_{k,\chi}}{k}\equiv \frac{1}{2}\imod 1$. If $N=p>2$ and $p\mid 1-\chi(t)t^k$, then $-\frac{B_{k,\chi}}{k}\equiv \frac{1-p}{p^{\nu+1}}\imod 1$, and the theorem follows. 
\end{proof}

\begin{Exa} We note that Lemma 4.1, hence Theorem 4.2, is independent of the sign vector $\epsilon$.
Consider the case when $N=13$ and $k=6$. We see that $2\imod 13$ is a primitive root and $\chi(2)=-1$. Hence by Theorem 4.2, $13\mid 2a_0^\epsilon(n)$ for any $n>0$ where $f_0^\epsilon=\sum_na_0^\epsilon(n)q^n$. The reduced modular form $f_0^+$ when $\epsilon=+1$ is 
\[f_0^+=\frac 12 - 26q^9 - 39q^{10} - 91q^{12} - 78q^{13} - 195q^{14} - 390q^{16} - 546q^{17} + O(q^{20}),\]
and $f_0^-$ when $\epsilon=-1$ is 
\[f_0^-=\frac 12+ 13q^7 + 13q^8 + 65q^{11} + 65q^{13} + 286q^{15} + 728q^{18} + 1001q^{19}+ O(q^{20}).\]
\end{Exa}

\begin{Rmk}
In the degenerate case $N=1$, the above divisibility extends Siegel's observation. For example, when $k=12$, we have $p\mid a_0(n)$. and $p=2,3,5,7,13$. More precisely, we should have that $2^3, 3^2, 5, 7$ and $13$ divide $a_0(n)$ by the well-known \emph{Staudt-Clausen theorem}. Note that
\[a_0(3)=16773120=2^{12}\cdot 3^2\cdot 5\cdot 7\cdot 13\quad\text{and}\quad a_0(4)=398034000=2^{4}\cdot 3^7\cdot 5^3\cdot 7\cdot 13.\]
\end{Rmk}

We are interested in the particular case when $N>1$ is a fundamental discriminant, $\chi=\left(\frac{N}{\cdot}\right)$, $k=0$, and $\epsilon$ is such that $\epsilon_p=\chi_p(-1)$. In Borcherds's theory of automorphic products, a weakly holomorphic modular form $f\in A^\epsilon(N,k,\chi)$ is lifted to a Hilbert modular form $\Psi_f$ for $\mathbb{Q}(\sqrt N)$ (see \cite[Theorem 13.3]{borcherds1998automorphic}, \cite[ Theorem 9]{bruinier2003borcherds}, and \cite[Theorem 4.1]{kim2013weakly}). 

\begin{Cor} Keep notations in Lemma 4.1 and assume $N=5$, $2-k=0$ and $\epsilon=+1$. Then the weight of $\Psi_{f_m}$ ($m<0$) is divisible by $5$.
\end{Cor}
\begin{proof}
Note $k=2$ and $t=2$ is a primitive root mod $5$, and one checks that $\chi(2)=-1$ and $(\frac{5-1}{2},2)=\frac{5-1}{2}$. So by Theorem 4.2 and Zagier duality (Theorem 2.2), we have $5\mid 2a_m(0)$ since $\nu=0$. Actually, one can show that $a_m(0)\in\mathbb Z$, hence $5\mid a_m(0)$, via the duality and Sturm's theorem as discussed at the beginning of this section. Then by \cite[Theorem 9]{bruinier2003borcherds} or \cite[Theorem 4.1]{kim2013weakly}, we have the weight of $\Psi_{f_m}$ is given by $s(0)a_m(0)/2=a_m(0)$, so $5$ divides the weight.
\end{proof}

\begin{Rmk}
The divisibility by $5$ in Corollary 4.5 can be seen numerically from Bruinier and Bundschuh's computation \cite{bruinier2003borcherds}. Such divisibility can also be proved from the Eisenstein series $f_0=E^\epsilon$ via the duality. This is the only divisibility we can obtain for the weights of $\Psi_{f_m}$, because $k=2$. Actually, by \cite[Theorem 3]{carlitz1959arithmetic}, if $N$ is composite, $L(-1,\chi)$ is integral. If $N=p>2$, then Theorem 4.2 implies that $\frac{p-1}{2}\mid 2$, hence $p\leq 5$.
\end{Rmk}

\subsection{Divisibility for Reduced Modular Forms with $k\leq 0$}
We now state some divisibility results in an ``orthogonal" direction in the sense of Zagier duality, generalizing the divisibility results by Duke and Jenkins \cite{duke2008zeros}. While Siegel's result and its generalization for $f_0$ takes place when $k\geq 2$, the divisibility we now consider is for reduced modular forms of weight $k\leq 0$. We shall try two ways of obtaining such divisibility results, one by applying Hecke operators and one by applying the differential operator $D^{1-k}$.

Recall that we denote $f_m^\epsilon=\sum_na_m^\epsilon(n)q^n$ the reduced modular form of order $m$, if it exists, in $A^\epsilon(N,k,\chi)$.

\begin{Thm}
Let $m,n\in\mathbb Z$ and $r$ be a positive integer with $(r,mnN)=1$. We assume $k\leq 0$ and for some $\epsilon$, \[S^{\epsilon^*}(N,2-k,\chi)=S^{(\epsilon')^*}(N,2-k,\chi)=\{0\}\] where $\epsilon'$ is determined by $\epsilon_p'=\epsilon_p\chi_p(r)$. Then if $f_{m}^\epsilon$ exists,
$r^{1-k}\mid s(mr)a_{mr}^{\epsilon'}(n)$.
\end{Thm}
\begin{proof}
Since $f_m^\epsilon$ exists for a fixed integer $m$ (necessarily negative), by Lemma 5.5 in \cite{zhang2013zagier} and the assumption on the cuspform spaces, we have $f_{mr}^{\epsilon'}$ exists.

Again by the assumption on the cuspform space, we see that any reduced modular form in $A^\epsilon(N,k,\chi)$ satisfies  $f_m=\frac{1}{s(m)}q^m+O(1)$ with $m<0$. By Lemma 3.3, $f_m^\epsilon|_kT(r)\in A^{\epsilon'}(N,k,\chi)$. Therefore, by the Hecke action, we have
\begin{align*}f_m^\epsilon|_kT(r)&=\sum_n\sum_{0<d\mid (r,n)}\chi(d)d^{k-1}a_m^\epsilon(rn/d^2)q^n\\
&=\sum_{0<d\mid r\mid md}\chi(d)d^{k-1}a_m^\epsilon(m)q^{\frac{md^2}{r}}+O(1)\\
&=\sum_{0<d\mid r\mid md}\chi(d)d^{k-1}\frac{s(md^2/r)}{s(m)}f_{md^2/r}^{\epsilon'}+O(1)\\
&=\sum_{0<d\mid r\mid md}\chi(d)d^{k-1}f_{md^2/r}^{\epsilon'},
\end{align*}
where the last equality follows from the fact that $M^{\epsilon'}(N,k,\chi)=\{0\}$ and the second last equality follows from the assumption $S^{(\epsilon')^*}(N,2-k,\chi)=\{0\}$. By comparing the $q^n$-coefficients, this implies that
\[\sum_{0<d\mid (r,n)}\chi(d)d^{k-1}a_m^\epsilon(rn/d^2)=\sum_{0<d\mid r\mid md}\chi(d)d^{k-1}a_{md^2/r}^{\epsilon'}(n).\]
We multiply both sides by $r^{1-k}$ and since $(r,mnN)=1$, we obtain
\[r^{1-k}a_{m}^\epsilon(rn)=a_{mr}^{\epsilon'}(n),\quad\text{or}\quad r^{1-k}s(m)a_{m}^\epsilon(rn)=s(mr)a_{mr}^{\epsilon'}(n).\] This finishes the proof.
\end{proof}

If $r\in R_0$, then $T(r)$ is an operator on $A^\epsilon(N,k,\chi)$ and we have the following corollary. In the following, we shall drop the sign vector $\epsilon$ in the notations.

\begin{Cor}
Let $f_m=\sum_na_m(n)q^n$ be a reduced modular form in $A^\epsilon(N,k,\chi)$ for some $\epsilon$.
Assume $k\leq 0$ and $S^{\epsilon^*}(N,2-k,\chi)=\{0\}$. If $r\in R_0$ and $n\in\mathbb Z$ with $(r,mnN)=1$, we have $r^{1-k}\mid s(mr)a_{mr}(n)$. 
\end{Cor}

If we want to consider each individual prime and then combine the divisibility, we have to apply operators $T(p^2)$ (unless $p\in R_0$), the generators of the Hecke algebra treated in \cite{bruinier2010weil}. In this case we can obtain the following theorem. Note that $r$ and $m/r$ may not be relative prime.

\begin{Thm}
Let $f_m=\sum_na_m(n)q^n$ be a reduced modular form in $A^\epsilon(N,k,\chi)$ for some $\epsilon$.
Assume $k\leq 0$ and $S^{\epsilon^*}(N,2-k,\chi)=\{0\}$. If $m=-\prod_{p}p^{m_p}$, set 
\[r=\prod_{p>2,p\nmid N}p^{r_p}, \quad \text{ with } r_p=\left\{\begin{matrix}
m_p & \text{if } 2\mid m_p,\\
\frac{m_p-1}{2} &\text{if } 2\nmid m_p.
\end{matrix}\right.\] 
Then for any $n$ with $(m,n)=1$, $r^{1-k}\mid s(m)a_{m}(n)$.
\end{Thm}
\begin{proof} The proof is similar to that of Theorem 4.7. For $k\leq 0$, $r\in R_0$ and $f_m=\frac{1}{s(m)}q^m+O(1)$ with $m<0$, we have  Therefore, by the Hecke action above,
\begin{align*}f_m|_kT(r)&=
\sum_{0<d\mid r\mid md}\chi(d)d^{k-1}\frac{s(md^2/r)}{s(m)}f_{md^2/r},
\end{align*}
where the last equality follows from the fact that $M^\epsilon(N,k,\chi)=\{0\}$. By comparing the $q^n$-coefficients, 
\[\sum_{0<d\mid (r,n)}\chi(d)d^{k-1}a_m(rn/d^2)=\sum_{0<d\mid r\mid md}\chi(d)d^{k-1}a_{md^2/r}(n).\]
We then only have to argue locally. Assume that $m=-\prod_pp^{m_p}$. Fix one $p\mid m$ with $p\nmid N$. Assume first that $m_p$ is even and let $r=p^{m_p}$ and $m'=m/r$. We apply the above equality to $m'$ and $r$ and multiply both sides by $r^{1-k}$. Since $(r,m'n)=1$, we obtain
\[r^{1-k}a_{m'}(rn)=a_{m'r}(n)\quad\text{or}\quad r^{1-k}s(m')a_{m'}(rn)=s(m'r)a_{m'r}(n).\]
Clearly we have $r^{1-k}\mid s(m'r)a_{m'r}(n)$. If $m_p$ is odd, then let $r=p^{\frac{m_p-1}{2}}$, $m'=m/r^2$ and apply $T(r^2)$. Similarly we have
\[r^{2(1-k)}a_{m'}(r^2n)=a_{m}(n)+\chi(p)p^{1-k}a_{m/p^2}(n),\] or
\[r^{2(1-k)}s(m')a_{m'}(r^2n)=s(m)a_{m}(n)+\chi(p)p^{1-k}s(m/p^2)a_{m/p^2}(n),\]
where if $p^2\nmid m$, the second term on the right side should be omitted.
By induction on $m_p$, we have $r^{1-k}=p^{\frac{(m_p-1)(1-k)}{2}}\mid s(m)a_{m}(n)$. This finishes the proof.
\end{proof}

\begin{Exa}
Consider the reduced modular forms in Example 2.4 ($N=8$, $k=0$ and $\epsilon=+1$). For $g_{-6}=\sum_nb_{-6}(n)q^n$, Theorem 4.7 applies and we see that $3\mid 2b_{-6}(n)$ if $3\nmid n$. One can also see the divisibility by $7$ of the Fourier coefficients of $g_{-7}$.
\end{Exa}

\begin{Exa}
Consider $f_{-4}$ in the introduction computed by Bruinier and Bundschuh \cite{bruinier2003borcherds} where $N=5$, $k=0$ and $\epsilon=+1$.
Actually all coefficients of $f_{-4}$ are integral from the construction. One can also prove this by computing Sturm's bound. One verifies that if $2\nmid n$, $4\mid a_{-4}(n)$, as predicted by the above theorems.
\end{Exa}

\begin{Exa}
Consider $N=15$, $k=-1$ again and keep notations in Example 3.4.  If $\epsilon_4=[1,1]$, we have
\[\arraycolsep=1.6pt\def\arraystretch{1.1}
\begin{array}{rrrrrrrrrrrr}
f_{-11}^{\epsilon_4}&=\sum_na_{-11}^{\epsilon_4}(n)q^n &=& q^{-4} &- 15& - 47q &+ 92q^4 &+ 498q^6 &- 543q^9  &+O(q^{10}).\\
\end{array}\]
The Fourier coefficients are not divisible by $11$, and the reason is that $(\epsilon'_4)^*=\epsilon_4$ (in Theorem 4.7) and $S^{\epsilon_4}(15,3,\chi)\neq \{0\}$. For another example, consider $r=2$ and $\epsilon_4$. We have $\epsilon_4'=\epsilon_1$ and the assumptions in Theorem 4.7 are valid. Therefore, if $(2,mn)=1$, then
\[2^2\mid s(2m)a_{2m}^{\epsilon_1}(n).\]
For example, this becomes $2\mid a_{-10}^{\epsilon_1}(n)$ if $n$ is odd, for the following reduced modular form:
\[f_{-10}^{\epsilon_1}=\frac{1}{2}q^{-10} - \frac{15}{2} + 45q^2 - 60q^3 + 68q^5 + 410q^8 - 1395q^{12} - 1584q^{15} + 5320q^{17} - 6870q^{18}+O(q^{20}).\]
\end{Exa}

In above theorems, we applied Hecke operators that either preserve or permute the $\epsilon$-subspaces, so the argument does not apply to the Hecke operators $U(p)$ with $p\mid N$. Now we end this paper with the divisibility result obtained by applying the differential operator $D^{1-k}$ where $D=q\frac{d}{dq}$. This includes the divisibility by $p\mid N$ as a special case.

\begin{Thm}
Let $f_{mp}=\sum_{n}a(n)q^n\in A^\epsilon(N,k,\chi)$ be a reduced modular form for $p$ prime, $m\in\mathbb Z$, $k\leq 0$ and a sign vector $\epsilon$. If
\[S^\epsilon(N,2-k,\chi)=S^{\epsilon^*}(N,2-k,\chi)=\{0\},\] then $p^{1-k}\mid s(mp)a(n)$ whenever $p\nmid n$.
\end{Thm}
\begin{proof}
From the assumption that $S^{\epsilon^*}(N,2-k,\chi)=\{0\}$, we see that $f_{mp}=\frac{1}{s(mp)}q^{mp}+O(q^{mp+1})$. By applying the differential operator $D^{1-k}$ with $D=q\frac{d}{dq}$, we see that
\[D^{1-k}f_{mp}=\sum_nn^{1-k}a(n)q^n\in A^\epsilon(N,2-k,\chi).\]
From the assumption that $S^\epsilon(N,2-k,\chi)=\{0\}$ and the fact that the constant term of $D^{1-k}f_{mp}$ vanishes, we must have $D^{1-k}f_{mp}=(mp)^{1-k}g_{mp}$; here $g_{mp}$ is the reduced modular form of order $mp$ in $A^\epsilon(N,2-k,\chi)$. It follows that the Fourier coefficients of $s(mp)D^{1-k}f_{mp}$ are divisible by $(mp)^{1-k}$. In particular, if $(p,n)=1$, $p^{1-k}\mid s(mp)a(n)$. We are done.
\end{proof}

If $S(N,2-k,\chi)=\{0\}$, Theorem 4.13 includes the previous theorems. And such argument also applies in the degenerate case $N=1$, simplifying the Hecke operator argument of Duke and Jenkins \cite{duke2008zeros}.

\begin{Cor}
If $f_{m}=\sum_{n}a_m(n)q^n\in A^\epsilon(N,k,\chi)$ is a reduced modular form for $k\leq 0$ and a sign vector $\epsilon$ and if $S(N,2-k,\chi)=\{0\}$, then $m^{1-k}\mid s(m)a_m(n)$ whenever $(m,n)=1$.
\end{Cor}

\begin{Exa}
In Bruinier and Bundschuh's example \cite{bruinier2003borcherds} ($N=5$, $\epsilon=+1$, $k=0$), we have $5\mid a_{-5}(n)$ for any $n$ with $(5,n)=1$ in the expansion of $f_{-5}$ in the introduction.
\end{Exa}

When $k\geq 2$, we can apply the Zagier duality and obtain corresponding divisibility results. 

\vskip 0.5 cm

\addcontentsline{toc}{chapter}{Bibliography}
\bibliographystyle{amsplain}
\bibliography{paper}

\end{document}